\documentclass[11pt]{amsart}
\usepackage{fullpage}
\usepackage{amssymb}
\usepackage{amsmath}

\newcommand{\ba}{\begin{array}}
\newcommand{\eea}{\end{eqnarray}}
\newcommand{\ea}{\end{array}}

\newtheorem{definition}{Definition}[section]
\newtheorem{theorem}[definition]{Theorem}
\newtheorem{lemma}[definition]{Lemma}
\newtheorem{proposition}[definition]{Proposition}

\usepackage[matrix,arrow,curve]{xy}
\usepackage[utf8x]{inputenc}
\usepackage{tikz}
\title{Symplectic A-directed immersions}
\author{Sauvik Mukherjee}
\begin{document}
\maketitle

\section{Introduction} Let $V$ and $W$ be smooth manifolds of dimensions $n$ and $q$ respectively with $n<q$ and let $\pi: Gr_n(W)\rightarrow W$ be the Grassmanian bundle of $n$-planes tangent to $W$, i.e $\pi^{-1}(x)=Gr_n(T_xW)$ for $x\in W$. For a given vector bundle monomorphism $F:TV\rightarrow TW$, define $GF:V\rightarrow Gr_nW$ given by \[x\longmapsto F_x(T_xV)\] Now for an immersion $f:V\rightarrow W$ we call $Gdf$ the tangential lift of $f$. Let $A\subset Gr_n(W)$ be an arbitrary subset then an immersion $f:V\rightarrow W$ is called A-directed if $Gdf(V)\subset A$. In \cite{MR864505} Gromov proved the following theorem using Convex Integration thechnique which is known as A-directed embedding.\\

\begin{theorem}(A-directed embedding)
 \label{GAE}
Let $V$ and $W$ be as above with $V$ an open manifold. Let $A\subset Gr_n(W)$ be an open set and $f_0:V\hookrightarrow W$ be an embedding such that the tangential lift $G_0=Gdf_0$ is homotopic to a map $G_1:V\rightarrow A\subset Gr_n(W)$ through a homotopy $G_t:V\rightarrow Gr_n(W)$ such that $\pi \circ G_t=f_0$, then $f_0$ can be isotoped to an A-directed embedding $f_1: V\rightarrow W$. Moreover given a core $K\subset V$, the isotopy $f_t$ can be chosen arbitrarily $C^0$-close to $f_0$ on $Op(K)$.
\end{theorem}
 
Alternative proof of the theorem has been provided in \cite{MR1909245} and \cite{MR1833750}. In this paper we have presented a proof of a symplectic analogue of this theorem. For $V$ closed \ref{GAE} is not true in general but under a completeness condition on $A$,  \ref{GAE} becomes true, see \cite{MR1909245} for a proof.\\

Let us conclude this section by stating the main theorem of this paper. Let $(V,\omega_V)$ and $(W,\omega_W)$ be symplectic manifolds of dimensions $2n$ and $2q$ respectively with $n<q$. Let $Sp_{2n}(W)\subset Gr_{2n}(W)$ consists of symplectic $2n$-planes tangential to $W$, and $\pi:Sp_{2n}(W)\rightarrow W$ be the bundle.\\

\begin{theorem}
 \label{SAI}
Let $(V,\omega_V=d\alpha_V)$ be an open exact symplectic manifold and $(W,\omega_W=d\alpha_W)$ be an exact symplectic manifold of dimensions $2n$ and $2q$ respectively with $n<q$. Let $A\subset Sp_{2n}(W)$ be an open subset and let $f_0:V\rightarrow W$ be a symplectic immersion i.e, $f_0$ is an immersion satisfying $f_0^*(\omega_W)=\omega_V$. Further assume that the tangential lift $G_0=Gdf_0:V \rightarrow Sp_{2n}(W)$ admits a homotopy $G_t:V \rightarrow Sp_{2n}(W)$ such that $\pi \circ G_t=f_0$ and $G_1:V\rightarrow A\subset Sp_{2n}(W)$ then $f_0$ admits a homotopy of isosymplectic immersions $f_t$ such that $f_1$ is A-directed. Moreover given a core $K\subset V$, the homotopy $f_t$ can be chosen arbitrarily $C^0$-close to $f_0$ on $Op(K)$.
\end{theorem}

\section{Main Theorem} In this section we will present a proof of \ref{SAI}. First we need an approximation result which we state bellow.\\

\begin{theorem}(Main Approximation Lemma)
 \label{MAL}
$(V,\omega_V=d\alpha_V)$ be an open exact symplectic manifold as above and let $K\subset V$ be a polyhedron of positive codimension and $G_t:V\rightarrow Sp_{2n}(W)$ be a symplectic tangential homotopy, where $(W,\omega_W=d\alpha_W)$ as above is exact symplectic, such that $G_0=Gdf$ for some isosymplectic immersion $f:V\rightarrow W$ and $\pi \circ G_t=f$. Then for arbitrarily small $\delta,\varepsilon >0$ there exists a $\delta$-small symplectic diffeotopy $h^{\tau}:V\rightarrow V$, $\tau \in [0,1]$ and a homotopy of symplectic immersions \[\tilde{f}_t:Op(\tilde{K})\rightarrow W\] where $\tilde{K}=h^1(K)$ and $\tilde{f}_0=(f_0)_{\mid Op_V\tilde{K}}$, such that the homotopy $Gd\tilde{f}_t$ is $\varepsilon$-close to the tangential homotopy $(G_t)_{\mid Op_V\tilde{K}}$. 
\end{theorem}

We need a theorem from \cite{MR1909245} for the proof of \ref{MAL} which involves the following concepts.\\

\begin{definition}(Local Integrability)
 \label{LI}
 A differential relation $\mathcal{R}\subset J^r(V,W)$ is called (parametrically) locally integrable if given a map $h:I^k\rightarrow V$, a family of sections \[F_p:h(p)\rightarrow \mathcal{R},\ p\in I^k\] and a family local holonomic extensions near $\partial I^k$ \[\tilde{F}_p:Op(h(p))\rightarrow \mathcal{R},\ \tilde{F}_p(h(p))=F_p(h(p)),\ p\in Op(\partial I^k)\] there exists a family of local holonomic extensions \[\tilde{F}_p:Op(h(p))\rightarrow \mathcal{R},\ \tilde{F}_p(h(p))=F_p(h(p)),\ p\in I^k\] such that for $p\in Op(\partial I^k)$ these new extensions coinside with the original extensions over $Op(h(p))$.
\end{definition}

\begin{definition}(Microflexibility)
 \label{MF}
 Let $K^m=[-1,1]^m$. For a fixed $n$ and any $k<n$ we denote by $\theta_k$ the pair $(K^n,K^k\cup \partial K^n)$. Let $V$ be an n-dimentional manifold. A pair $(A,B)\subset V$ is called a $\theta_k$ pair if $(A,B)$ is diffeomorphic to the standard pair $\theta_k$.\\
 
 A differential relation $\mathcal{R}\subset J^r(V,W)$ is called (parametrically) $k$-microflexible if for any sufficiently small open ball $U\subset V$ and any families parametrized by $p\in I^m$ of 
 \begin{itemize}
  \item $\theta_k$-pairs $(A_p,B_p)\subset U$,\\
  \item holonomic sections $F^0_p:Op(A_p)\rightarrow \mathcal{R}$ and \\
  \item holonomic homotopies $F^{\tau}_p:Op(B_p)\rightarrow \mathcal{R},\ \tau\in [0,1]$, of the sections $F^0_p$ over $Op(B_p)$ which are constant over $Op(\partial B_p)$ for all $p\in I^m$ and constant over $Op(B)$ for $p\in Op(\partial I^m)$\\
 \end{itemize}
there exists a number $\sigma >0$ and a family of holonomic homotopies \[F^{\tau}_p:Op(A_p)\rightarrow \mathcal{R},\ \tau\in [0,\sigma],\] which extend the family of homotopies \[F^{\tau}_p:Op(B_p)\rightarrow \mathcal{R},\ \tau\in [0,\sigma],\] and are constant over $Op(\partial A_p)$ for all $p\in I^m$ and constant over $Op(A)$ for $p\in Op(\partial I^m)$.\\

A differential relation $\mathcal{R}\subset J^r(V,W)$ is called microflexible if $\mathcal{R}$ is $k$-microflexible for all $k=0,1,2,...,n-1$, where $n=dim(V)$.
\end{definition}

\begin{definition}\cite{MR1909245}
 \label{CAP}
 Let $\mathcal{U}\subset Diff(V)$ be a lie subgroup of the group of compactly supported diffeomorphisms of $V$, and $\mathcal{A}$ be its lie algebra. We call $\mathcal{U},\ (\mathcal{A})$ "capacious" if it satisfies the following conditions\\
 
 \begin{itemize}
  \item for any $v\in \mathcal{A}$, any compact set $A\subset V$ and its neighborhood $U\supset A$ there exists a vector field $\tilde{v}_{A,U}\in \mathcal{A}$ which is supported in $U$ and which coinsides with $v$ on $A$.\\
  \item given any tangent hyperplane $\tau \subset T_xV,\ x\in V$, there exists a vector field $v\in \mathcal{A}$ which is transversal to $\tau$.\\
 \end{itemize}
 Moreover the above conditions need to be satisfied parametrically with respect to a compact parameter space.
 \end{definition}

It has been mentioned in \cite{MR1909245} that the relation of Legendrian immersions is locally integrable and microflexible. Let us prove this. We need a lemma from \cite{MR1909245} to prove this.\\

\begin{lemma}(Contact Stability)
 \label{CS}
Let $\xi_t,\ t\in I$, be a family of contact structures on a neighborhood $Op(A)\subset M$ of a compact set $A\subset M$. Then there exists an isotopy of $\phi_t:Op(A)\rightarrow M$, fixed on $A$ such that $\phi_t^*\xi_0=\xi_t,\ t\in I$.
\end{lemma}

\begin{lemma}
 \label{LIMFL}
 Let $M$ be any manifold of dimension $n$ and $(N,\xi)$ be a contact manifold with dimension $2n+1$, then $\mathcal{R}_{Leg}$, the relation of Legendrian immersions from $M$ to $N$ is locally integrable and microflexible.
\end{lemma}

\begin{proof}
 To avoid notational complexity, we shall only prove the non-parametric versions.\\
 
 {\bf Local Integrability:} Take $F:T_xM\rightarrow \xi_y\subset T_yN$ be an injective linear map. So now take an immersion $f:Op(x)\rightarrow Op(y)$ such that $df_x=F$. Take a hyper-plane field $\xi'$ on $Op(y)$ such that $df(TOp(x))\subset \xi'$, $\xi' \pitchfork \xi^{\perp}$ and $\xi'_y=\xi_y$. As $\xi'_y=\xi_y$, $\xi'$ is also a contact structure on $Op(y)$. Let $\xi=ker(\eta),\ and\ \xi'=ker(\eta')$, where $\eta\ and\ \eta'$ are one forms on $Op(y)$. We can also assume $\eta_y=\eta'_y$. Let $\eta_t=(1-t)\eta+t\eta',\ t\in I$. So $(\eta_t)_y=\eta_y$. So $\xi_t=ker(\eta_t)$ is a homotopy of contact structures on $Op(y)$. So by \ref{CS}, we get an isotopy $\phi_t:Op(y)\rightarrow Op(y)$, such that $(\phi_t)^*\xi_0=\xi_t$. As $\xi_0=\xi$ and $\xi_1=\xi'$, we get $(\phi_1)^*\xi=\xi'$. So the required Legendrian immersion is $\phi_1 \circ f$.\\  

 {\bf Microflexibility:} Let $U,\ (A,B),\ F^{\tau}$ be as in \ref{MF}. As the relation of immersions is microflexible, we get a $\tilde{\sigma}>0$ such that there exists a homotopy \[\tilde{f}^{\tau}:Op(A)\rightarrow N,\ \tau \in [0,\tilde{\sigma}]\] which extends the homotopy \[f^{\tau}:Op(B)\rightarrow N,\ where\ F^{\tau}=df^{\tau}\] Now take a homotopy $\xi_{\tau}$ of hyperplane fields, such that $d\tilde{f}^{\tau}(TOp(A))\subset \xi_{\tau}\ and\ \xi_0=\xi$. As $U$ is small, $\xi_{\tau}=ker\eta_{\tau}$ for some homotopy of one forms $\eta_{\tau}$. Define \[P:[0,\tilde{\sigma}]\rightarrow \Omega^{2n+1}(N),\ as\ P(\tau)=\eta_{\tau}\wedge (d\eta_{\tau})^n\] So $P(0)$ is a non-zero section in $\Omega^{2n+1}(N)$. So there exists a $\sigma \in (0,\tilde{\sigma}]$ such that $\eta_{\tau}$ are contact forms for $\tau \in [0,\sigma]$. Again by \ref{CS} we get an isotopy $\phi_{\tau}$ such that $(\phi_{\tau})^*ker(\eta_0)=ket(\eta_{\tau})$. So set \[f^{\tau}:Op(A)\rightarrow N\] as $\phi_{\tau}\circ \tilde{f}^{\tau},
\ for\ \tau \in [0,\sigma]$.
 \end{proof}

\begin{theorem}(\cite{MR1909245})
 \label{MHAT}
 Let $\mathcal{R}\subset X^{(r)}$ be a locally integrable, microflexible differential relation. $K\subset V$ be a polyhedron of positive codimension and $F_z:OpK\rightarrow \mathcal{R}$ be a family of sections parametrized by a cube $I^m,\ m=0,1,2,...$. Suppose the sections $F_z$ are holonomic for $z\in Op(\partial I^m)$. Then for arbitrarily small $\delta, \varepsilon >0$ there exists a family of $\delta$-small diffeotopies $h^{\tau}_z:V\rightarrow V,\ \tau\in I,\ z\in I^m$ and a family of holonomic sections $\tilde{F}_z: Op(h^1_z(K))\rightarrow \mathcal{R},\ z\in I^m$ such that \\
 
 \begin{itemize}
  \item $h^{\tau}_z=id_V$ and $\tilde{F}_z=F_z$ for all $z\in Op(\partial I^m)$\\
  \item $dist(\tilde{F}_z(v),(F_z)_{\mid Op(h^1_z(K))}(v))< \varepsilon$ for all $v\in Op(h^1_z(K))$\\
 \end{itemize}
 
 Moreover if $\mathcal{R}$ is $\mathcal{U}$-invariant, where $\mathcal{U}\subset Diff(V)$ is a "capacious" subgroup, then $h^{\tau}_z$ can be taken from $\mathcal{U}$, i.e, we can assume $h^{\tau}_z\in \mathcal{U}$.
 \end{theorem}
 
 \begin{proposition}
  \label{RAT}
    Let $K\subset V$ be a polyhedron of positive codimension and $F_t:TV\rightarrow TW$ be a homotopy of vector bundle monomorphisms covering an isosymplectic immersion $f$ with $F_0=df,\ (f-isosymplectic)$ such that $F_t^*{\omega_W}=\omega_V$. Then for arbitrarily small $\delta,\varepsilon >0$, there exists a $\delta$-small symplectic diffeotopy $h^{\tau}:V\rightarrow V$ and a homotopy of isosymplectic immersions $\tilde{f}_t:Op(\tilde{K})\rightarrow W$, where $\tilde{K}=h^1(K)$ and $\tilde{f}_0=f_{\mid Op(\tilde{K})}$ such that the homotopy $Gd\tilde{f}_t:Op(\tilde{K})\rightarrow Sp_2n(W)$ is $\varepsilon$-close to $(GF_t)_{\mid Op(\tilde{K})}$.
 \end{proposition}

\begin{proof}
 Define $F'_t$ as \[F'_t:=F_{2t},\ t\in[0,1/2]\ and\ F'_t:=F_{2-2t},\ t\in[1/2,1]\]
 
 As $f:V\rightarrow W$ is an isosymplectic immersion, we consider the following symplectic vector bundle $E\rightarrow V$ whose fiber over a point $v\in V$ is the space $(df(T_vV))^{\perp_{\omega_W}}$, the $\omega_W$-dual to $df(T_vV)\subset T_{f(v)}W$. By (9.2.2) of \cite{MR1909245} there exists a symplectic structure $\tilde{\omega}$ on a neighborhood $Op(V)$ of the zero section $V$ of $E$ such that $\tilde{\omega}_{\mid V}=\omega_V$. As both $\omega_V\ and\ \omega_W$ are exact, $\tilde{\omega}$ can also be taken to be exact, so let $\tilde{\omega}=d\tilde{\alpha}$. Now $F'_t$ can be extended to a homotopy of vector bundle morphism $F''_t:TOp(V)\rightarrow TW$ which fiber wise isomorphism, such that $(F''_t)^*\omega_W=\tilde{\omega}$. Moreover $F''_0\ and\ F''_1$ can be be chosen to be holonomic by "Symplectic Neighborhood Theorem" (9.3.2) of \cite{MR1909245}. Now consider the symplectic manifold $Op(V)\times W$ with symplectic structure $\tilde{\omega}\oplus (-\omega_W)$. Define \[F^0_t:TOp(V)\rightarrow 
TOp(V)\times TW\] \[X \longmapsto (X,F''_t(X))\] Observe that for each fixed $t\in I,\ and\ x\in Op(V)$, $Im(F^0_t)_x$ is an isotropic subspace.\\
 
  Now take $s:I\rightarrow I$ such that $s(t)=0,\ for\ t\in [0,\sigma]$, and $s(t)=1,\ for\ t\in [1-\sigma,1],\ where\ 0<\sigma<1/2\ is\ small$ and $s_{\mid[\sigma,1-\sigma]}:[\sigma,1-\sigma]\rightarrow I$ is a homeomorphism. Let $s^{-1}(1/2)=t_0$. Define $\tilde{F}_t$ as $\tilde{F}_t:=F^0_{s(t)}$. Observe that $\tilde{F}_t$ is holonomic for $t\in Op(\partial I)$.\\

Consider the contact manifold $Op(V)\times W\times \mathbb{R}$ with contact structure $ker(dz-\eta)$, where $\eta=\tilde{\alpha}\oplus (-\alpha_W)$ and $z$ is the variable in $\mathbb{R}$. The homotopy of isotropic monomorphisms $\tilde{F}_t$ lifts to a homotopy of isotropic monomorphisms \[L(\tilde{F}_t):TOp(V)\rightarrow T(Op(V)\times W\times \mathbb{R})\] which in our case is a homotopy of Legendrian monomorphisms.\\

 Now in view of \ref{LIMFL} we can use \ref{MHAT}. So by \ref{MHAT} we get, for given arbitrarily small $\delta, \varepsilon >0$ there exists a family of $\delta$-small diffeotopies $h^{\tau}_t:Op(V)\rightarrow Op(V),\ \tau\in I,\ t\in I$ and a family of holonomic sections $\tilde{F}^{hol}_t: Op(h^1_t(K))\rightarrow \mathcal{R}_{Leg},\ t\in I$ such that \\
 
 \begin{itemize}
  \item $h^{\tau}_t=id_{Op(V)}$ and $\tilde{F}^{hol}_t=L(\tilde{F}_t)$ for all $t\in Op(\partial I)$\\
  \item $dist(\tilde{F}^{hol}_t(v),L(\tilde{F}_t)_{\mid Op(h^1_t(K))}(v))< \varepsilon$ for all $v\in Op(h^1_t(K))$\\
 \end{itemize}

 Moreover, as $\mathcal{U}=Ham(Op(V))$, the identity component of the group of compactly supported hamiltonian diffeomorphisms of the symplectic manifold $(Op(V),d\tilde{\alpha})$ is "capacious" and $\mathcal{R}_{Leg}$ here is invariant under the action of $\mathcal{U}$, we can take $h^{\tau}_t\in \mathcal{U}$.\\
 
  Consider $\tilde{F}^{hol}_t,\ for\ t\in [0,t_0]$ to get a homotopy of Legendrian immersions which we denote by $f^{hol}_t:Op(h^1_t(K))\rightarrow Op(V)\times W\times \mathbb{R},\ for\ t\in [0,t_0]$. Let $\pi:Op(V)\times W\times \mathbb{R}\rightarrow Op(V)\times W$ be the projection on the first two factors. Then $\pi \circ f^{hol}_t:Op(h^1_t(K))\rightarrow Op(V)\times W$ is a homotopy of exact Lagrangian immersions into $(Op(V)\times W,d\eta)$. \\

 Set $\tilde{K}=h^1_{t_0}(K)\subset h^1_{t_0}(V)$. As $h^1_{t_0}$ is a diffeomorphism, we identify $V$ with $h^1_{t_0}(V)$. Now take $\tilde{v}\in Op_V(\tilde{K})$. So there exists a unique $v\in Op_V(K)$ such that $\tilde{v}=h^1_{t_0}(v)$. Now consider the curve \[\gamma_{\tilde{v}}:[0,t_0]\rightarrow Op(V)\times W,\ \gamma_{\tilde{v}}(t)=\pi \circ f^{hol}_{t_0-t}(h^1_{t_0-t}(v))\] Set $\tilde{f}'_t(\tilde{v})=\gamma_{\tilde{v}}(t_0-t),\ for\ t\in[0,t_0]$. So $\tilde{f}'_t(\tilde{v})=(\tilde{v},\tilde{f}_t(\tilde{v}))$. As the restriction of a isotropic immersion is isotropic, so $(\tilde{f}_t)_{\mid V}$ is isosymplectic. Now the required homotopy of isosymplectic immersions is given by reparametrizing $(\tilde{f}_t)_{\mid V}$.
\end{proof}

\begin{proof}(Proof of \ref{MAL})
 For $\bar{\varepsilon}\in (0,\pi/4)$ choose an integer $N$ such that for each interval \[\Delta_i=[(i-1)/N,i/N]\] the homotopy $\{G_t\}_{t\in \Delta_i}$ is $\bar{\varepsilon}$-small. Set $K_0=K\ and\ V_0=Op(K_0)$.\\
 
 {\bf Step-1:} $\{G_t\}_{t\in \Delta_1}$ defines a homotopy of sections \[F^1_t:V_0\rightarrow \mathcal{R}_{iso-symp}\subset J^1(V_0,W)\] such that base of $F_t$ is $f$ for all $t\in [0,1/N]$. By \ref{RAT} above one can $\varepsilon_1$-approximate $F^1_t$ by $J^1\tilde{f}_t$ over $Op(h^1_1(K_0))$, where $h^1_1$ is a $\delta/N$-small symplectomorphism and $\tilde{f}_t:V_0\rightarrow W\ for\ t\in[0,1/N]$ are isosymplectic immersions. Set $K_1=h^1_1(K_0),\ V_1=Op(K_1)$.\\
 
 {\bf Step-2:} As $\varepsilon_1$ was chosen small, we can approximate $\{G_t\}_{t\in \Delta_2}$ by $\{G^2_t\}_{t\in \Delta_2}$ such that $\{G^2_t\}_{t\in \Delta_2}$ covers $\tilde{f}_{1/N}$ and $G^2_{1/N}=G(d\tilde{f}_{1/N})$. So $\{G^2_t\}_{t\in \Delta_2}$ defines a homotopy of sections \[F^2_t:V_1\rightarrow \mathcal{R}_{iso-symp}\subset J^1(V_1,W)\] such that base of $F^2_t$ is $\tilde{f}_{1/N}$. Again by \ref{RAT} one can $\varepsilon_2$-approximate $F^2_t$ by $J^1\tilde{f}_t$ over $Op(h^1_2(K_1))$, where $h^1_2$ is a $\delta/N$-small symplectomorphism and $\tilde{f}_t:V_1\rightarrow W\ for\ t\in[1/N,2/N]$ are isosymplectic immersions. Set $K_2=h^1_2(K_1),\ V_2=Op(K_2)$. Continue this way till $i=N$. Set $\tilde{K}=K_N$. Now we can define the required $f_t$ in the following way. For $t\in \Delta_i$ set \[f_t:Op(K_N)\rightarrow W\] as follows\\
 
 Take $v\in Op(K_N)$. So there is an unique $v'\in Op(K_i)$ such that \[h^1_N\circ...\circ h^1_{i+1}(v')=v\] So set $f_t(v)=\tilde{f}_t(v')$.
\end{proof}

We now state a theorem which will help us complete the proof of \ref{SAI}.\\

\begin{theorem}(\cite{MR2519215})
\label{DI}
 Let $(V,\omega_V)$ be an open symplectic manifold and let $K\subset V$ be a core of it, then there exists a homotopy of isosymplectic immersions $g_t:V\rightarrow V$ such that \[g_0=id_V,\ and\ g_1(V)\subset Op(K)\]
\end{theorem}

\begin{proof}(Proof of \ref{SAI})
 Let $K\subset V$ be a core. Use \ref{MAL} to approximate $G_t$ near $\tilde{K}=h^1(K)$ by a homotopy of isosymplectic immersions $\tilde{f}_t:Op_V(\tilde{K})\rightarrow W$. As $A\subset Sp_{2n}(W)$ is open, a sufficiently close approximation will give us $Gd\tilde{f}_1(Op_V(\tilde{K}))\subset A$. Now by \ref{DI} there exists a homotopy of isosymplectic immersions $g_t$ with the above properties with respect to the core $\tilde{K}$. So the required homotopy is $f_t=\tilde{f}_t\circ g_1$.
\end{proof}

\section{Generalization To Closed Manifolds} In this section we generalize the symplectic $A$-directed immersion theorem to closed manifolds. We introduce the following notions.\\
 
 \begin{definition}
  \label{CC}
  Let $n<m\leq q$. An open set $A\subset Sp_{2n}(W)$ is called $m$-complete if there exists an open set $\hat{A}\subset Sp_{2m}(W)$ such that \[A=\cup_{\hat{L}\in \hat{A}}Sp_{2n}(\hat{L})\]
 \end{definition}
 
\begin{lemma}
 \label{FP}
 Let $Sp_{m,n}(W)$ be the manifold of all $(2m,2n)$-symplectic flags on $W$, i.e, \[Sp_{m,n}(W)=\{(\hat{L},L):\hat{L}\in Sp_{2m}(W)\ and\ L\in Sp_{2n}(\hat{L})\}\] Consider the natural projection $Sp_{m,n}(W)\stackrel{P}{\rightarrow}Sp_{2n}(W)$ and let the following data be given 
   \[
  \xymatrix@=2pc@R=2pc{
 X\times \{0\}\ar@{->}[r]^-{f}\ar@{->}[d] & Sp_{m,n}(W)\ar@{->}[d]^-{P}\\
 X\times I \ar@{->}[r]_-{F} & Sp_{2n}(W)
 }
 \]
 Such that $\pi \circ F=\pi \circ P \circ f$ then $F$ lifts to a map $\tilde{F}:X\times I \rightarrow Sp_{m,n}(W)$, where $\pi:Sp_{2n}(W)\rightarrow W$ be the bundle projection.
\end{lemma}
 
 \begin{proof}
  Let $F(x,0)\subset T_yW$ for some $y\in W$. So $F(x,0)$ is a symplectic subspace of $T_yW$. Denote its symplectic complement by $G_x$. Now if $F$ is such that $F(x,t)\cap G_x= \{0\},\ for\ all\ t\ and\ all\ x$. Then defining $\tilde{F}(x,t)=(F(x,t)\oplus K_x,F(x,t))$ for some $2(m-n)$-dimensional symplectic subspace $K_x$ of $G_x$ provides a lift.\\
 
 Now let us consider the general case. Subdivide the parameter interval $I$ into subintervals $\Delta_i=[i/N,(i+1)/N]$ so that on each $\Delta_i$, $F(x,t)$ is $\delta$-small, where $\delta$ is a small positive real number. Let us choose $2(m-n)$-dimensional symplectic subspaces of the symplectic complements of $F(x,i/N)$ and denote it by $G(x,i/N)$ with $f(x,0)=(F(x,0)\oplus G(x,0),F(x,0))$. Let $s$ be a homeomorphism from $\Delta_i$ to $I$ such that $s(i/N)=0$ and $s((i+1)/N)=1$. Define for $t\in \Delta_i$ \[G(x,t)=\{(1-s(t))v+s(t)v',\ where\ v\in G(x,i/N)\ and\ v'\in G(x,(i+1)/N)\}\] Observe that as $F(x,t)$ is $\delta$-small for $t\in \Delta_i$ so is $G(x,t)$ for $t\in \Delta_i$. So for $t\in \Delta_i$, $G(x,t)$ consists of symplectic subspaces only. Now $\tilde{F}(x,t)=(F(x,t)\oplus G(x,t),F(x,t))$ provides the lift.
  \end{proof}

 \begin{theorem}
  \label{GCM}
  Let $A\subset Sp_{2n}(W)$ be an open set which is $m$-complete. Then the statement of \ref{SAI} holds for closed $V$.
 \end{theorem}

 \begin{proof}
  Let $Sp_{m,n}(W)$ be the manifold of all $(2m,2n)$-symplectic flags on $W$. Further let us denote the natural projections by \[Sp_{m,n}(W)\stackrel{\hat{P}}\rightarrow Sp_{2m}(W)\ and\ Sp_{m,n}(W)\stackrel{P}\rightarrow Sp_{2n}(W)\] Set \[\bar{A}=\{(\hat{L},L):\hat{L}\in \hat{A}, L\in Sp_{2n}(\hat{L})\}\subset Sp_{m,n}(W)\] In the above $\hat{A}$ exists by $m$-completeness. Note that $\hat{P}(\bar{A})=\hat{A},\ P(\bar{A})=A$. Let $G_t:V\rightarrow Sp_{2n}(W)$ be the homotopy between the tangential lift $G_0=Gdf_0$ of the isosymplectic immersion $f_0$ and the map $G_1:V\rightarrow A$. Let the map $G_1$ lifts to a map $\bar{G}_1:V\rightarrow \bar{A}\subset Sp_{m,n}(W)$. Then the homotopy $G_t$ lifts to a homotopy $\bar{G}_t:V\rightarrow Sp_{m,n}(W)$ by \ref{FP}. We have $G_t=P \circ \bar{G}_t$. 
Set $\hat{G}_t=\hat{P}\circ \bar{G}_t,\ t\in I$. Let $N$ be the total space of the vector bundle over $V$ whose fiber over a point $v\in V$ is the normal space to $G_1(v)$ in $\hat{G}_1(v)$. \\

Now extend the isosymplectic immersion $f_0$ to an immersion $F:Op_N(V)\rightarrow W$ such that $GdF_{\mid V}=(\hat{G}_0)_{\mid V}$. Consider $F^*(\omega_W)$, observe that $F^*(\omega_W)_{\mid V}=\omega_V$. Let $\tilde{\omega}$ be the symplectic structure on $Op_N(V)$ as in (9.3.2) \cite{MR1909245}. As $Op_N(V)$ is arbitrarily small we can assume that the linear homotopy $\omega_t$ joining $F^*(\omega_W)\ and\ \tilde{\omega}$ consist of symplectic structures only. So by symplectic stability theorem (9.3.2) \cite{MR1909245} there exists an isotopy $\phi_t:Op_N(V)\rightarrow Op_N(V)$ such that $\phi_t^*(F^*(\omega_W))=\omega_t\ and\ (\phi_t)_{\mid V}=id_V$. Set $\tilde{f}_0=F\circ \phi_1$. So we have constructed an isosymplectic immersion $\tilde{f}_0$ extending $f_0\ such\ that\ (Gd\hat{f}_0)_{\mid V}=(\hat{G}_0)_{\mid V}$. Hence by \ref{SAI} we can construct $\hat{f}_t:Op_N(V)\rightarrow W$ such that $\hat{f}_1$ is an $\hat{A}$-directed immersion. Then $f_t=(\hat{f}_t)_{\mid V}$ is the required one.\\

Now in general $G_1$ does not lift to $\bar{G}_1$. By \ref{MAL} we can assume that the homotopy $G_t$ is constant on a neighborhood $Op(K)$ of the $(2n-1)$-skeleton of some triangulation of $V^{2n}$. So we only need to construct $f_t$ on the top simplex $\Delta$ of the triangulation keeping $f_t$ fixed on $Op(\partial \Delta)$. If the triangulation is sufficiently small then $(G_1)_{\mid \Delta}$ lifts to $\bar{G}_1:\Delta \rightarrow \bar{A}$ and one can apply the previous argument.  
 \end{proof}
 
 \section{Application To Poisson Geometry} Let $(V^{2n},\omega_V=d\alpha_V)$ be a symplectic manifold and $(W^{2q},\omega_W=d\alpha_W)$ be another symplectic manifold such that there is a codimension $2k$ symplectic foliation $\mathcal{F}$ on $W$ with $q>n>k$. By a symplectic foliation we mean that $(\omega_W^{q-k})_{\mid T\mathcal{F}}$ is nowhere vanishing. Such a $W$ admits a regular Poisson structure. The foliation is known as the characteristic foliation. In this section we study the existence of a regular Poisson structures on a symplectic manifold $(V^{2n},\omega_V=d\alpha_V)$ with codimension $2k$ characteristic foliation induced from the symplectic structure $\omega_V$.\\
 
 \begin{theorem}
  \label{LT}
  Let $(V^{2n},\omega_V=d\alpha_V)$, $(W^{2q},\omega_W=d\alpha_W)$ and $\mathcal{F}$ be as above. Set \[A=\{L\in Sp_{2n}(W): L\pitchfork T\mathcal{F}\}\] Let there exists an isosymplectic immersion $f:V\rightarrow W$ whose tangential lift is homotopic to a map $G_1:V\rightarrow A$ through a homotopy $G_t:V\rightarrow Sp_{2n}(W)$ such that $\pi \circ G_t=f$. Then a regular Poisson structure on $V$ exists with codimension $2k$ characteristic foliation which is induced from $\omega_V$ if either $V$ is open or $A$ is $m$-complete for some $m$.  
 \end{theorem}

 \begin{proof}
  Observe that $A$ is open. So by \ref{SAI} or by \ref{GCM} there exists an isosymplectic immersion $f_1:V\rightarrow W$ such that $f_1\pitchfork \mathcal{F}$. So $f_1^*\mathcal{F}$ is a symplectic foliation on $(V,\omega_V)$. 
 \end{proof}

{\bf Acknowledgement.} I would like to thank Dr. M. Datta for her suggestions and D. Pancholi for a major Clarification.


\begin{thebibliography}{99}
 \bibitem{MR2519215} Datta, M. and Islam, M.R. (2009). Submersions on open symplectic manifolds. Topology Appl., 156(10):1801-1806.
 \bibitem{MR1909245} Eliashberg, Y. and Mishachev, N. (2002). Introduction to h-principle, volumn 48 of Graduate Studies in mathematics. American Mathematical Society, Providence, RI.
 \bibitem{MR864505} Gromov, M. (1986). partial differential relations, volumn 9 of Ergebnisse der Mathematik und ihrer Grenzgebiete (3) [Results in Mathematics and Related Areas (3)] Springer-Verlag, Berlin.
 \bibitem{MR1833750} Rourke, C. and Sanderson, B. (2001). The compression theorem II. Directed embeddings. Geom. topol., 5:431-440. 
\end{thebibliography}
\end{document}